\documentclass[twoside]{article}
\usepackage[]{aistats2019}

\usepackage{wrapfig}
\usepackage[utf8]{inputenc} 
\usepackage[T1]{fontenc}    
\usepackage{hyperref}       
\usepackage{url}            
\usepackage{booktabs}       
\usepackage{amsfonts}       
\usepackage{nicefrac}       
\usepackage{microtype}      
\usepackage{mystyle,amssymb,dsfont,amsmath,amsthm,mathtools}
\usepackage{notations}
\usepackage{xcolor}
\usepackage{subcaption}

\usepackage[round]{natbib}

\graphicspath{{./images/}}

\begin{document}

\twocolumn[

\aistatstitle{Sample Complexity of Sinkhorn Divergences}

\aistatsauthor{ Aude Genevay \And L\'enaic Chizat \And Francis Bach \\ \And Marco Cuturi \And  Gabriel Peyr\'e}

\aistatsaddress{ DMA,\\ENS Paris \And  INRIA \And INRIA and\\ DI, ENS Paris \And  Google and \\ CREST ENSAE \And CNRS and DMA, \\ ENS Paris } ]

\begin{abstract}
\vskip-.2cm
Optimal transport (OT) and maximum mean discrepancies (MMD) are now routinely used in machine learning to compare probability measures. 
We focus in this paper on \emph{Sinkhorn divergences} (SDs), a regularized variant of OT distances which can interpolate, depending on the regularization strength $\varepsilon$, between OT ($\varepsilon=0$) and MMD ($\varepsilon=\infty$). Although the tradeoff induced by that regularization is now well understood computationally (OT, SDs and MMD require respectively $O(n^3\log n)$, $O(n^2)$ and $n^2$ operations given a sample size $n$), much less is known in terms of their \emph{sample complexity}, namely the gap between these quantities, when evaluated using finite samples \emph{vs.} their respective densities. Indeed, while the sample complexity of OT and MMD stand at two extremes, $O(1/n^{1/d})$ for OT in dimension $d$ and $O(1/\sqrt{n})$ for MMD, that for SDs has only been studied empirically. In this paper, 
we \emph{(i)} derive a bound on the approximation error made with SDs when approximating OT as a function of the regularizer $\varepsilon$, \emph{(ii)} prove that the optimizers of regularized OT are bounded in a Sobolev (RKHS) ball independent of the two measures and \emph{(iii)} provide the first sample complexity bound for SDs, obtained,by reformulating SDs as a maximization problem in a RKHS. We thus obtain a scaling in $1/\sqrt{n}$ (as in MMD), with a constant that depends however on $\varepsilon$, making the bridge between OT and MMD complete.
\end{abstract}

\section{Introduction}
Optimal Transport (OT) has emerged in recent years as a powerful tool to compare probability distributions. Indeed, Wasserstein distances can endow the space of probability measures with a rich Riemannian structure~\citep{ambrosio2006gradient}, one that is able to capture meaningful geometric features between measures even when their supports do not overlap. OT has been, however, long neglected in data sciences for two main reasons, which could be loosely described as \emph{computational} and \emph{statistical}: computing OT is costly since it requires solving a network flow problem; and suffers from the curse-of-dimensionality, since, as will be made more explicit later in this paper, the Wasserstein distance computed between two samples converges only very slowly to its population counterpart.

Recent years have witnessed significant advances on the computational aspects of OT. A recent wave of works have exploited entropic regularization, both to compare  discrete measures with finite support~\citep{CuturiSinkhorn} or measures that can be sampled from~\citep{2016-genevay-nips}. Among the many learning tasks performed with this regularization, one may cite domain adaptation~\citep{courty2014domain}, text retrieval~\citep{kusner2015word} or multi-label classification \citep{2015-Frogner}. The ability of OT to compare probability distributions with disjoint supports (as opposed to the Kullback-Leibler divergence) has also made it  popular as a loss function to learn generative models \citep{WassersteinGAN,OTGAN,ABC}.

At the other end of the spectrum, the maximum mean discrepancy (MMD)~\citep{GrettonMMD} is an integral probability metric~\citep{empmmd} on a reproducing kernel Hilbert space (RKHS) of test functions. The MMD is easy to compute, and has also been used in a very wide variety of applications, including for instance the estimation of generative models \citep{li2015generative,MMD-GAN,MMDGAN}.

OT and MMD differ, however, on a fundamental aspect: their sample complexity. The definition of sample complexity that we choose here is the convergence rate of a given metric between a measure and its empirical counterpart, as a function of the number of samples. This notion is crucial in machine learning, as bad sample complexity implies overfitting and high gradient variance when using these divergences for parameter estimation. In that context, it is well known that the sample complexity of MMD is independent of the dimension, scaling as $\frac{1}{\sqrt{n}}$ \citep{GrettonMMD} where $n$ is the number of samples.  In contrast, it is well known that standard OT suffers from the curse of dimensionality \citep{dudley1969speed}: Its sample complexity is exponential in the dimension of the ambient space. Although it was recently proved that this result can be refined to consider the implicit dimension of data \citep{weed2017sharp}, the sample complexity of OT appears now to be the major bottleneck for the use of OT in high-dimensional machine learning problems. 

A remedy to this problem may lie, again, in regularization. Divergences defined through regularized OT, known as Sinkhorn divergences, seem to be indeed less prone to over-fitting. Indeed, a certain amount of regularization seems to improve performance in simple learning tasks \citep{CuturiSinkhorn}. Additionally, recent papers \citep{ramdas2017wasserstein,genevay2018learning} have pointed out the fact that Sinkhorn divergences are in fact interpolating between OT (when regularization goes to zero) and MMD (when regularization goes to infinity). However, aside from a recent central limit theorem in the case of measures supported on discrete spaces~\citep{bigot2017central}, the convergence of empirical Sinkhorn divergences, and more generally their sample complexity, remains an open question.

\paragraph{Contributions.}
This paper provides three main contributions, which all exhibit theoretical properties of Sinkhorn divergences. Our first result is a bound on the speed of convergence of regularized OT to standard OT as a function of the regularization parameter, in the case of continuous measures.  The second theorem proves that the optimizers of the regularized optimal transport problem lie in a Sobolev ball which is independent of the measures. This allows us to rewrite the Sinkhorn divergence as an expectation maximization problem in a RKHS ball and thus justify the use of kernel-SGD for regularized OT as advocated in \citep{2016-genevay-nips}. As a consequence of this reformulation, we provide as our third contribution a sample complexity result. We focus on how the sample size and the regularization parameter affect the convergence of the empirical Sinkhorn divergence (i.e., computed from samples of two continuous measures) to the continuous Sinkhorn divergence. We show that the Sinkhorn divergence benefits from the same sample complexity as MMD, scaling in $\frac{1}{\sqrt{n}}$ but with a constant that depends on the inverse of the regularization parameter. Thus sample complexity worsens when getting closer to standard OT, and there is therefore a tradeoff between a good approximation of OT (small regularization parameter) and fast convergence in terms of sample size (larger regularization parameter). We conclude this paper with a few numerical experiments to asses the dependence of the sample complexity on $\epsilon$ and $d$ in very simple cases.

\paragraph{Notations.} We consider $\Xx$ and $\Yy$ two bounded subsets of $\RR^d$ and we denote by $\abs{\Xx}$ and $\abs{\Yy}$ their respective diameter $\sup \{ \norm{x-x'} | x, x' \in \Xx (resp. \Yy) \}$. The space of positive Radon measures of mass 1 on $\Xx$ is  denoted $\Mm_+^1(\Xx)$ and we use upper cases $X,Y$ to denote random variables in these spaces. We use the notation $\phi = O(1+ x^k)$ to say that $\phi \in \RR$ is bounded by a polynomial of order $k$ in $x$ with positive coefficients. 

\section{Reminders on Sinkhorn Divergences}

We consider two probability measures $\alpha \in \Mm_+^1(\Xx)$ and $\beta$ on $\Mm_+^1(\Yy)$. The \citeauthor{Kantorovich42} formulation~\citeyearpar{Kantorovich42} of optimal transport between $\alpha$ and $\beta$ is defined by
\eql{W(\alpha,\beta) \eqdef \min_{\pi \in \Pi(\alpha,\beta)} \int_{\Xx \times \Yy} c(x,y) \d\pi(x,y) \tag{$\Pp$},
}
where the feasible set is composed of probability distributions over the product space $\Xx \times \Yy$ with fixed marginals $\alpha,\beta$:
\eq{ 
	\Pi(\alpha,\beta) \eqdef \enscond{\pi \in \Mm_+^1(\Xx \times \Yy) }{ P_{1\sharp}\pi=\alpha, P_{2\sharp}\pi=\beta  },
} 
where $P_{1\sharp}\pi$ (resp.~$P_{2\sharp}\pi$) is the marginal distribution of $\pi$ for the first (resp.~second) variable, using the projection maps $P_1(x,y)=x; P_2(x,y)=y$ along with the push-forward operator $_\sharp$.

The cost function $c$ represents the  cost to move a unit of mass from $x$ to $y$. Through this paper, we will assume this function to be $\Cc^\infty$ (more specifically, we need it to be $\Cc^{\frac{d}{2}+1}$). When $\Xx = \Yy$ is endowed with a distance~$d_\Xx$, choosing  $c(x,y)=d_\Xx(x,y)^p$ where $p \geq 1$ yields the $p$-Wasserstein distance  between probability measures.

We introduce regularized optimal transport, which consists in adding an entropic regularization to the optimal transport problem, as proposed in~\citep{CuturiSinkhorn}. Here we use the relative entropy of the transport plan with respect to the product measure $\alpha \otimes \beta$ following \citep{2016-genevay-nips}:
\begin{align}\label{OTreg}
W_\epsilon(\alpha,\beta) \eqdef \!\min_{\pi\in\Pi(\alpha,\beta)} \int_{\Xx \times \Yy}\!\!\! &c(x,y) \d\pi(x,y) \nonumber \\&+ \epsilon H(\pi \mid \alpha\otimes \beta) \tag{$\Pp_\epsilon$},
\end{align}
where 
\eql{H(\pi \mid \alpha\otimes \beta) \eqdef  \int_{\Xx \times \Yy} \log\left(\frac{\d\pi(x,y)}{\d\alpha(x)\d\beta(y)}\right) \d\pi(x,y). \label{entropy}}

Choosing the relative entropy as a regularizer allows to express the dual formulation of regularized OT as the maximization of an expectation problem, as shown in \citep{2016-genevay-nips}
\begin{align*} \label{dual OTreg}
W_\epsilon(\alpha,\beta) = & \max_{u \in \Cc(\Xx),v \in \Cc(\Yy)} \int_\Xx u(x) \d\alpha(x) + \int_\Yy v(y) \d\beta(y) \\
& - \epsilon \int_{\Xx \times \Yy} e^{\frac {u(x)+v(y) - c(x,y)}{\epsilon}}\d\alpha(x) \d\beta(y)  + \epsilon \\
= & \max_{u \in \Cc(X),v \in \Cc(Y)} \EE_{\alpha\otimes\beta} \left[f_\epsilon^{XY}(u,v) \right] + \epsilon
\end{align*}
where $f_\epsilon^{xy}(u,v) = u(x) + v(y) - \epsilon  e^{\frac {u(x)+v(y) - c(x,y)}{\epsilon}}.$
This reformulation as the maximum of an expectation will prove crucial to obtain sample complexity results. The existence of optimal dual potentials $(u,v)$ is proved in the appendix. They are unique $\alpha-$ and $\beta-$a.e. up to an additive constant.

To correct for the fact that $W_\epsilon(\alpha,\alpha)\ne 0$,  \citep{genevay2018learning} propose Sinkhorn divergences, a natural normalization of that quantity defined as
\eql{ \bar W_\epsilon(\alpha,\beta) = W_\epsilon(\alpha,\beta) - \frac{1}{2} (W_\epsilon(\alpha,\alpha) + W_\epsilon(\beta,\beta)).
\label{SinkhornDiv}}
This normalization ensures that $\bar W_\epsilon(\alpha,\alpha) = 0$, but also has a noticeable asymptotic behavior as mentioned in~\citep{genevay2018learning}. Indeed, when $\epsilon \rightarrow 0$ one recovers the original (unregularized) OT problem, while choosing $\epsilon \rightarrow +\infty$ yields the maximum mean discrepancy associated to the kernel $k = -c/2$, where MMD is defined by:
\begin{align*}
	MMD_k(\alpha,\beta) = \mathbb{E}_{\alpha\otimes\alpha}[k(X,X')] &+ \mathbb{E}_{\beta\otimes \beta}[k(Y,Y')] \\
	&- 2 \mathbb{E}_{\alpha\otimes \beta}[k(X,Y)].
\end{align*}
In the context of this paper, we study in detail the sample complexity of $W_\epsilon(\alpha,\beta)$, knowing that these results can be extended to $\bar W_\epsilon(\alpha,\beta)$.
\section{Approximating Optimal Transport with Sinkhorn Divergences}

In the present section, we are interested in bounding the error made when approximating $W(\alpha,\beta)$ with $W_\epsilon(\alpha,\beta)$. 

\begin{thm}
Let $\alpha$ and $\beta$ be probability measures on $\Xx$ and $\Yy$ subsets of $\mathbb{R}^d$ such that $\abs{\Xx} = \abs{\Yy} \leq D$ and assume that $c$ is $L$-Lipschitz w.r.t.\ $x$ and $y$. It holds
\begin{eqnarray}
0 \leq W_\epsilon(\alpha,\beta) - W(\alpha,\beta) &\leq  2 \epsilon d \log\left(\frac{e^2\cdot L\cdot D}{\sqrt{d}\cdot \epsilon}\right) \\  &\sim_{\epsilon \to 0 } 2\epsilon d \log(1/\epsilon) .
\end{eqnarray}
\end{thm}

\begin{proof}
For a probability measure $\pi$ on $\Xx\times \Yy$, we denote by $C(\pi)=\int c \, \d \pi$ the associated transport cost and by $H(\pi)$ its relative entropy  with respect to the product measure $\alpha \otimes \beta$ as defined in \eqref{entropy}.
Choosing $\pi_0$ a minimizer of $\min_{\pi \in \Pi(\alpha,\beta)} C(\pi)$, we will build our upper bounds using a family of transport plans with finite entropy that approximate $\pi_0$. The simplest approach consists in considering block approximation. In contrast to the work of~\citet{Carlier2017}, who also considered this technique, our focus here is on quantitative bounds.

\begin{defn}[Block approximation]
For a resolution $\Delta>0$, we consider the block partition of $\RR^d$ in hypercubes of side $\Delta$ defined as 
\begin{eqnarray*}
 \{ Q^\Delta_k = {[k_1\cdot \Delta,(k_1+1)\cdot \Delta[} \times \dots {[k_d\cdot \Delta,(k_d +1)\cdot \Delta[}\; ;\\
  k = (k_1,\dots,k_d) \in \mathbb{Z}^d\}.
\end{eqnarray*}

To simplify notations, we introduce $Q^\Delta_{ij}\eqdef Q^\Delta_i \times Q^\Delta_j$, $\boldsymbol{\alpha}_i^\Delta \eqdef \alpha(Q_i^\Delta)$, $\boldsymbol{\beta}_j^\Delta \eqdef \beta(Q_j^\Delta)$.
The block approximation of $\pi_0$ of resolution $\Delta$ is the measure $\pi^\Delta \in \Pi(\alpha,\beta)$ characterized by
\[
\pi^\Delta\vert_{Q^\Delta_{ij}} =  \frac{\pi_0(Q^\Delta_{ij})}{\boldsymbol{\alpha}_i^\Delta \cdot \boldsymbol{\beta}_j^\Delta } (\alpha\vert_{Q^\Delta_i}\otimes \beta\vert_{Q^\Delta_j})
\]
for all $(i,j) \in (\mathbb{Z}^d)^2$, with the convention $0/0=0$.
\end{defn}

$\pi^\Delta$ is nonnegative by construction. Observe also that for any Borel set $B\subset \RR^d$, one has
\begin{align*}
\pi^\Delta(B\times \RR^d) &= \sum_{(i,j)\in (\mathbb{Z}^d)^2} \frac{\pi_0(Q_{ij}^\Delta)}{\boldsymbol{\alpha}_i^\Delta \cdot \boldsymbol{\beta}_j^\Delta } \cdot \alpha(B\cap Q^\Delta_i) \cdot \boldsymbol{\beta}_j^\Delta \\
&= \sum_{i \in \mathbb Z^d} \alpha(B\cap Q^\Delta_i) =  \alpha(B),
\end{align*}
which proves, using the symmetric result in $\beta$, that $\pi^\Delta$ belongs to $\Pi(\alpha,\beta)$. 
As a consequence, for any $\epsilon>0$ one has $W_\epsilon(\alpha,\beta)  \leq C(\pi^\Delta) + \epsilon H(\pi^\Delta)$.
Recalling also that the relative entropy $H$ is nonnegative over the set of probability measures, we have the bound
\[
0\leq W_\epsilon(\alpha,\beta) - W(\alpha,\beta) \leq (C(\pi^\Delta)-C(\pi_0)) + \epsilon H(\pi^\Delta).
\]
We can now bound the terms in the right-hand side, and choose a value for $\Delta$ that minimizes these bounds.

The bound on $C(\pi^\Delta)-C(\pi_0)$ relies on the Lipschitz regularity of the cost function. Using the fact that $\pi^\Delta(Q^\Delta_{ij})=\pi_0(Q^\Delta_{ij})$ for all $i,j$, it holds
\begin{align*}
C(\pi^\Delta)-C(\pi_0) 
& = \sum_{(i,j)\in (\mathbb{Z}^d)^2}  \pi_0(Q^\Delta_{ij}) \Big( \sup_{x,y \in Q^\Delta_{ij}} c(x,y) \\  & \phantom{\sum_{(i,j)\in (\mathbb{Z}^d)^2}  \pi_0(Q^\Delta_{ij}) \Big(} - \inf_{x,y \in Q^\Delta_{ij}} c(x,y) \Big)\\
&\leq 2L\Delta \sqrt{d},
\end{align*}
where $L$ is the Lipschitz constant of the cost (separately in $x$ and $y$) and $\Delta\sqrt{d}$ is the diameter of each set $Q^\Delta_{i}$.

As for the bound on $H(\pi^\Delta)$, using the fact that $\pi_0(Q^\Delta_{ij})\leq 1$ we get
\begin{align*}
H(\pi^\Delta) &= \! \! \sum_{(i,j)\in (\mathbb{Z}^d)^2}  \! \log\left(\frac{ \pi_0(Q^\Delta_{ij})}{\boldsymbol{\alpha}_i^\Delta\cdot \boldsymbol{\beta}_j^\Delta} \right) \pi_0(Q^\Delta_{ij}) \\
& \leq \sum_{(i,j)\in (\mathbb{Z}^d)^2} \left( \log(1/\boldsymbol{\alpha}_i^\Delta) + \log(1/ \boldsymbol{\beta}_j^\Delta) \right)  \pi_0(Q^\Delta_{ij})\\
& = - H^\Delta(\alpha) - H^\Delta(\beta),
\end{align*}
where we have defined $H^\Delta(\alpha) =\sum_{i \in \mathbb{Z}^d} \boldsymbol{\alpha}_i^\Delta\log(\boldsymbol{\alpha}_i^\Delta)$ and similarly for $\beta$. Note that in case $\alpha$ is a discrete measure with finite support, $H^\Delta(\alpha)$ is equal to (minus) the discrete entropy of $\alpha$ as long as $\Delta$ is smaller than the minimum separation between atoms of $\alpha$. However, if $\alpha$ is not discrete then $H^\Delta(\alpha)$ blows up to $-\infty$ as $\Delta$ goes to $0$ and we need to control how fast it does so.
Considering $\alpha^\Delta$ the block approximation of $\alpha$ with constant density $\boldsymbol{\alpha}_i^\Delta/\Delta^d$ on each block $Q^\Delta_i$ and (minus) its differential entropy $H_{\mathcal{L}^d}(\alpha^\Delta) = \int_{\mathbb{R}^d} \alpha^\Delta(x) \log \alpha^\Delta(x) dx$, it holds 
$ H^\Delta(\alpha) = H_{\mathcal{L}^d}(\alpha^{\Delta}) - d\cdot \log(1/\Delta)$. Moreover, using the convexity of $H_{\mathcal{L}^d}$, this can be compared with the differential entropy of the uniform probability on a hypercube containing $\Xx$ of size $2D$. Thus it holds $H_{\mathcal{L}^d}(\alpha^{\Delta}) \geq - d \log (2D)$ and thus $H^\Delta(\alpha) \geq - d \cdot \log(2D/\Delta)$.

Summing up, we have for all $\Delta>0$
\[
W_\epsilon(\alpha,\beta) - W(\alpha,\beta) \leq  2L \Delta \sqrt{d} + 2\epsilon d\cdot \log(2D/\Delta) .
\]
The above bound is convex in $\Delta$, minimized with $\Delta = 2\sqrt{d}\cdot \epsilon/L$. This yields
\[
W_\epsilon(\alpha,\beta) - W(\alpha,\beta) \leq 4\epsilon d + 2 \epsilon d \log\left(\frac{L\cdot D}{\sqrt{d}\cdot \epsilon}\right).\qedhere
\]

\end{proof}
\section{Properties of Sinkhorn Potentials}

We prove in this section that Sinkhorn potentials are bounded in the Sobolev space $\mathbf H^s(\RR^d)$ regardless of the marginals $\alpha$ and $\beta$. For $s>\frac{d}{2}$, $\mathbf H^s(\RR^d)$ is a reproducing kernel Hilbert space (RKHS): This property will be crucial to establish sample complexity results later on, using standard tools from RKHS theory.

\begin{defn} The Sobolev space $\mathbf H^s(\Xx)$, for $s \in \NN^*$, is the space of functions $\phi : \Xx \subseteq \RR^d \rightarrow \RR$ such that for every multi-index $k$ with $\abs{k} \leq s$ the mixed partial derivative $\phi^{(k)}$ exists and belongs to $L^2(\Xx)$. It is endowed with the following inner-product
\eql{ \langle \phi, \psi \rangle _{\mathbf H^s(\Xx)} = \sum_{\abs{k} \leq s} \int_{\Xx} \phi^{(k)}(x) \psi^{(k)}(x) \d x .
}
\end{defn}

\begin{thm}\label{thm_rkhs}
When $\Xx$ and $\Yy$ are two compact sets of $\RR^d$ and the cost $c$ is $\Cc^\infty$, then the Sinkhorn potentials $(u,v)$ are uniformly bounded in the Sobolev space $\mathbf H^s(\RR^d)$ and their norms satisfy
\eq{
\norm{u}_{\mathbf H^s} = O\left(1+ \frac{1}{\epsilon^{s-1}}\right) \; \text{and} \;  \norm{v}_{\mathbf H^s}= O\left(1+\frac{1}{\epsilon^{s-1}}\right),} 
with constants that only depend on $\abs{\Xx}$ (or $\abs{\Yy}$ for $v$),$d$, and $\norminf{c^{(k)}}$ for $k = 0 ,\dots,  s$.
In particular, we get the following asymptotic behavior in $\epsilon$: $\norm{u}_{\mathbf H^s} = O(1)$  as $\epsilon \rightarrow +\infty$  and $\norm{u}_{\mathbf H^s} = O(\frac{1}{\epsilon^{s-1}})$ as $\epsilon \rightarrow 0$.

\end{thm}

To prove this theorem, we first need to state some regularity properties of the Sinkhorn potentials.
\begin{prop} \label{prop_potentials}
If $\Xx$ and $\Yy$ are two compact sets of $\RR^d$ and the cost $c$ is $\Cc^\infty$, then
\begin{itemize}
\item $u(x)\in [\min_y v(y) - c(x,y), \max_y v(y) - c(x,y)]$ for all $x\in \Xx$
\item $u$ is L-Lipschitz, where L is the Lipschitz constant of $c$
\item $u \in \Cc^\infty(\Xx)$ and $\norminf{u^{(k)}} =  O (1 + \frac{1}{\epsilon^{k-1}})$
\end{itemize}
and the same results also stand for $v$ (inverting $u$ and $v$ in the first item, and replacing $\Xx$ by $\Yy$).
\end{prop}

\begin{proof}
The proofs of all three claims exploit the optimality condition of the dual problem:
\eql{\label{optim_condition}
\!\!\!\exp\left(\frac{-u(x)}{\epsilon}\right) = \int \exp\left(\frac{v(y)-c(x,y)}{\epsilon}\right)\beta(y) \d y.
}
Since $\beta$ is a probability measure, $e^{\frac{-u(x)}{\epsilon}}$ is a convex combination of $\phi : x \mapsto e^{\frac{v(x)-c(x,y)}{\epsilon}}$ and thus $e^{\frac{-u(x)}{\epsilon}} \in [ \min_y \phi (y) , \max_y \phi (y)].$ We get the desired bounds by taking the logarithm. The two other points use the following lemmas:

\begin{lem}\label{lem_pot_1}
The derivatives of the potentials are given by the following recurrence
\begin{equation}
u^{(n)}(x) =  \int g_n(x,y) \gamma_\epsilon(x,y) \beta(y) \d y,
\end{equation}
where $$g_{n+1} (x,y) = g_n' (x,y) + \frac{u'(x)-c'(x,y)}{\epsilon} g_n (x,y) ,$$ $g_1(x,y) = c'(x,y)$
and $\gamma_\epsilon(x,y) = \exp(\frac{u(x)+v(y)-c(x,y)}{\epsilon})$.
\end{lem}

\begin{lem}\label{lem_pot_2} The sequence of auxiliary functions $(g_k)_{k=0\dots}$ verifies
 $\norminf{u^{(k)}} \leq \norminf{g_k}$. Besides,  for all $j=0,\dots,k$, for all $k=0,\dots,n-2$, $\norminf{g_{n-k}^{(j)}}$ is bounded by a polynomial in $\frac{1}{\epsilon}$ of order $n-k+j-1$.
\end{lem}

The detailed proofs of the lemmas can be found in the appendix. We give here a sketch in the case where $d=1$. Lemma~\ref{lem_pot_1} is obtained by a simple recurrence, consisting in differentiating both sides of the dual optimality condition. Differentiating under the integral is justified with the usual domination theorem, bounding the integrand thanks to the Lipschitz assumption on $c$, and this bound is integrable thanks to the marginal constraint. Differentiating once and rearranging terms gives:
\begin{equation} \label{uprim}
u'(x) = \int c'(x,y) \gamma_\epsilon(x,y) \beta(y) dy.
\end{equation}
where $\gamma_\epsilon$ is defined in Lemma~\ref{lem_pot_1}. One can easily see that $\gamma_\epsilon'(x,y) = \frac{u'(x)-c'(x,y)}{\epsilon} \gamma_\epsilon(x,y)$ and this allows to conclude the recurrence, by differentiating both sides of the equality. From the primal constaint, we have that $ \int_\Yy  \gamma_\epsilon(x,y) \beta(y) \d y = 1$. Thus thanks to Lemma~\ref{lem_pot_1} we immediately get that $\norminf{u^{(n)}} \leq \norminf{g_n}$. For $n=1$, since $g_1 = c'$ we get that $\norminf{u'} = \norminf{c'} = L$ and this proves the second point of Proposition~\ref{prop_potentials}. For higher values of $n$, we need the result from Lemma~\ref{lem_pot_2}. This property is also proved by recurrence, but requires a bit more work. To prove the induction step, we need to go from bounds on $g_{n-k}^{(i)}$, for $k=0,\dots,n-2$ and $i=0,\dots,k$ to bounds on $g_{n+1-k}^{(i)}$, for $k=0,\dots,n-1$ and $i=0,\dots,k$. 
Hence only new quantities that we need to bound are  $g_{n+1-k}^{(k)}, k=0,\dots,n-1$. This is done by another (backwards) recurrence on $k$ which involves some tedious computations, based on Leibniz formula, that are detailed in the appendix.
\end{proof}

Combining the bounds of the derivatives of the potentials with the definition of the norm in $\mathbf H^s$, is enough to complete the proof of Theorem~\ref{thm_rkhs}.
\begin{proof}(Theorem~\ref{thm_rkhs}) The norm of $u$ in $\mathbf H^s(\Xx)$ is
\eq{\norm{u}_{\mathbf H^s} =  \left( \sum_{\abs{k} \leq s} \int_\Xx (u^{(k)})^2  \right)^\frac{1}{2} \leq \abs{\Xx} \left(\sum_{\abs{k} \leq s} \norminf{u^{(k)}}^2 \right)^\frac{1}{2}.}
From Proposition~\ref{prop_potentials} we have that $\forall k, \norminf{u^{(k)}} = O (1+\frac{1}{\epsilon^{k-1}})$ and thus we get that $\norm{u}_{\mathbf H^s} = O (1+\frac{1}{\epsilon^{s-1}})$.
We just proved the bound in $\mathbf H^s(\Xx)$ but we actually want to have a bound on $\mathbf H^s(\RR^d)$. This is immediate thanks to the Sobolev extension theorem \citep{calderon1961lebesgue} which guarantees that $\norm{u}_{\mathbf H^s(\RR^d)} \leq C \norm{u}_{\mathbf H^s(\Xx)}$ under the assumption that $\Xx$ is a bounded Lipschitz domain. 
\end{proof}

This result, aside from proving useful in the next section to obtain sample complexity results on the Sinkhorn divergence, also proves that kernel-SGD can be used to solve continuous regularized OT. This idea introduced in \cite{2016-genevay-nips} consists in assuming the potentials are in the ball of a certain RKHS, to write them as a linear combination of kernel functions and then perform stochastic gradient descent on these coefficients. Knowing the radius of the ball and the kernel associated with the RKHS (here the Sobolev or Matérn kernel) is crucial to obtain good numerical performance and ensure the convergence of the algorithm.

\section{Approximation from Samples}

In practice, measures $\alpha$ and $\beta$ are only known through a finite number of samples. Thus, what can be actually computed in practice is the Sinkhorn divergence between the empirical measures $\hat \alpha_n \eqdef \frac{1}{n} \sumin \delta_{X_i}$ and $\hat \beta_n \eqdef \frac{1}{n} \sumin \delta_{Y_i}$, where $(X_1,\dots,X_n)$ and $(Y_1,\dots,Y_n)$ are n-samples from $\alpha$ and $\beta$, that is
\begin{multline*}
W_\epsilon(\hat\alpha_n,\hat\beta_n)  =  \max_{u,v} \sumin u(X_i) + \sumin v(Y_i) \\
- \epsilon \sumin \exp\left(\frac {u(X_i)+v(Y_i) - c(X_i,Y_i)}{\epsilon}\right) + \epsilon \\
 =   \max_{u,v} \frac{1}{n}\sumin f_\epsilon^{X_iY_i} (u,v) + \epsilon,
\end{multline*}
where $(X_i,Y_i)_{i=1}^n$ are i.i.d random variables distributed according to $\alpha\otimes\beta$. On actual samples, these quantities can be computed using Sinkhorn's algorithm~\citep{CuturiSinkhorn}.

Our goal is to quantify the error that is made by approximating $\alpha,\beta$ by their empirical counterparts $\hat \alpha_n,\hat \beta_n$, that is bounding the following quantity:
\begin{align}
\vert W_\epsilon(\alpha,\beta) &- W_\epsilon(\hat \alpha_n, \hat \beta_n) \vert =  \nonumber\\
&\vert  \EE f_\epsilon^{XY}(u^*,v^*) - \frac 1 n  \sumin f_\epsilon^{X_iY_i} (\hat u , \hat v) \vert ,  \label{sup_bound}
\end{align}
where $(u^*,v^*)$ are the optimal Sinkhorn potentials associated with $(\alpha,\beta)$ and $(\hat u , \hat v)$ are their empirical counterparts.

\begin{thm} Consider the Sinkhorn divergence between two measures $\alpha$ and $\beta$ on $\Xx$ and $\Yy$ two bounded subsets of $\RR^d$, with a $\Cc^\infty$, $L$-Lipschitz cost $c$. One has 
\eq{\EE\vert W_\epsilon(\alpha,\beta) - W_\epsilon(\hat \alpha_n, \hat \beta_n) \vert  =   O \left(\frac{e^\frac{\kappa}{\epsilon}}{\sqrt{n}}\left(1+\frac{1}{\epsilon^{\lfloor d/2 \rfloor }}\right) \right) }
where $\kappa = 2 L \abs{\Xx}+\norminf{c}$ and constants only depend on $\abs{\Xx}$,$\abs{\Yy}$,$d$, and $\norminf{c^{(k)}}$ for $k = 0 \dots \lfloor d/2 \rfloor$. In particular, we get the following asymptotic behavior in $\epsilon$: 
\begin{align*}
\EE\vert W_\epsilon(\alpha,\beta) - W_\epsilon(\hat \alpha_n, \hat \beta_n) \vert  &=& O\left(\frac{e^\frac{\kappa}{\epsilon}}{\epsilon^{\lfloor d/2 \rfloor }\sqrt{n}}\right) \text{ as }\epsilon \rightarrow 0 \\
\EE\vert W_\epsilon(\alpha,\beta) - W_\epsilon(\hat \alpha_n, \hat \beta_n) \vert  &=& O\left(\frac{1}{\sqrt{n}}\right)\qquad\text{ as }\epsilon \rightarrow + \infty.
\end{align*}
\label{sample_complexity_thm}
\end{thm}

An interesting feature from this theorem is the fact when $\epsilon$ is large enough, the convergence rate does not depend on $\epsilon$ anymore. This means that at some point, increasing $\epsilon$ will not substantially improve convergence. However, for small values of $\epsilon$ the dependence is critical.

We prove this result in the rest of this section. The main idea is to exploit standard results from PAC-learning in RKHS. Our theorem is an application of the following result from \cite{bartlett2002rademacher} ( combining Theorem~12,4) and Lemma~22 in their paper):
\begin{prop}{(Bartlett-Mendelson '02)} \label{prop:Bartlett}
Consider $\alpha$ a  probability distribution, $\ell$ a B-lipschitz loss and $\Gg$ a given class of functions. Then
$$
\EE_\alpha \left[ \sup_{g \in \Gg} \EE_\alpha \ell(g,X) - \frac 1 n \sumin \ell(g,X_i) \right] \leq 2 B \EE_\alpha \Rr(\Gg(X_1^n))
$$
where $\Rr(\Gg(X_1^n))$ is the Rademacher complexity of class $\Gg$ defined by $ \Rr(\Gg(X_1^n))=\sup_{g \in \Gg} \EE_\sigma \frac 1 n  \sumin \sigma_i g(X_i)$ where $(\sigma_i)_i$ are iid Rademacher random variables.
Besides, when $\Gg$ is a ball of radius $\lambda$ in a RKHS with kernel $k$ the Rademacher complexity is bounded by
$$
\Rr(\Gg_\lambda(X_1^n)) \leq \frac \lambda n \sqrt{\sumin k(X_i,X_i)}.
$$
\end{prop}

Our problem falls in this framework thanks to the following lemma:
\begin{lem} \label{lem_bound} Let $\Hh^s_\lambda \eqdef \{ u \in \mathbf H^s(\RR^d) \mid \norm{u}_{\mathbf H^s(\RR^d)} \leq \lambda \}$, then there exists $\lambda$ such that:
\begin{align*}
\vert W_\epsilon(\alpha,\beta) &- W_\epsilon(\hat \alpha_n, \hat \beta_n) \vert \leq  \\
&3  \sup_{(u,v) \in (\Hh^s_\lambda)^2}\vert \EE f_\epsilon^{XY}(u,v) -\frac 1 n   \sumin f_\epsilon^{X_iY_i} (u, v) \vert .
\end{align*} 
\end{lem}

\begin{proof} Inserting  $\EE f_\epsilon^{XY}(\hat u,\hat v)$ and using the triangle inequality in \eqref{sup_bound} gives
\begin{align*}
	& \vert W_\epsilon(\alpha,\beta) \!-\! W_\epsilon(\hat \alpha_n, \hat \beta_n) \vert \leq \vert  \EE f_\epsilon^{XY}(u^*,v^*) \!-\! \EE f_\epsilon^{XY}(\hat u,\hat v)   \vert \\  
	 &\quad + \vert   \EE f_\epsilon^{XY}(\hat u,\hat v) -\frac 1 n   \sumin f_\epsilon^{X_iY_i} (\hat u, \hat v)   \vert .
\end{align*}
From Theorem~\ref{thm_rkhs}, we know that the all the dual potentials are bounded in $\mathbf H^s(\RR^d)$ by a constant $\lambda$ which doesn't depend on the measures. Thus the second term is bounded by $\sup_{(u,v) \in (\Hh^s_\lambda)^2} \vert \EE f_\epsilon(u,v) -\frac 1 n   \sumin f_\epsilon (u, v) \vert $ .\\
The first quantity needs to be broken down further. Notice that it is non-negative since $(u^*,v^*)$ is the maximizer of $\EE f_\epsilon(\cdot,\cdot)$  so we can leave out the absolute value. We have:
\begin{align}
\EE f_\epsilon^{XY}&(u^*,v^*) - \EE f_\epsilon^{XY}(\hat u,\hat v)   \leq  \nonumber\\
&\phantom{+}\EE f_\epsilon^{XY}(u^*,v^*) - \frac 1 n   \sumin f_\epsilon^{X_iY_i} (u^*,v^*)  \label{A}\\
& +\frac 1 n   \sumin f_\epsilon^{X_iY_i} (u^*,v^*)  - \frac 1 n   \sumin f_\epsilon^{X_iY_i} (\hat u, \hat v)  \label{B}\\
&+  \frac 1 n   \sumin f_\epsilon^{X_iY_i} (\hat u, \hat v) - \EE f_\epsilon^{XY}(\hat u,\hat v)  \label{C}
\end{align}
Both \eqref{A} and \eqref{C} can be bounded by $\sup_{(u,v) \in (\Hh^s_\lambda)^2} \vert \EE f_\epsilon^{XY}(u,v) -\frac 1 n   \sumin f_\epsilon^{X_iY_i}  (u, v) \vert $ while \eqref{B} is non-positive since $(\hat u,\hat v)$ is the maximizer of $\frac 1 n   \sumin f_\epsilon^{X_iY_i}  (\cdot , \cdot)$.
\end{proof}

To apply Proposition~\ref{prop:Bartlett} to Sinkhorn divergences we need to prove that \emph{(a)} the optimal potentials are in a RKHS and \emph{(b)} our loss function $f^\epsilon$ is Lipschitz in the potentials.

The first point has already been proved in the previous section. The RKHS we are considering is $\mathbf H^s(\RR^d)$ with $s=\lfloor\frac{d}{2} \rfloor+1$. It remains to prove that $f^\epsilon$ is Lipschitz in $(u,v)$ on a certain subspace that contains the optimal potentials.

 \begin{figure*}[h]
  \centering
  \includegraphics[width=.32\textwidth]{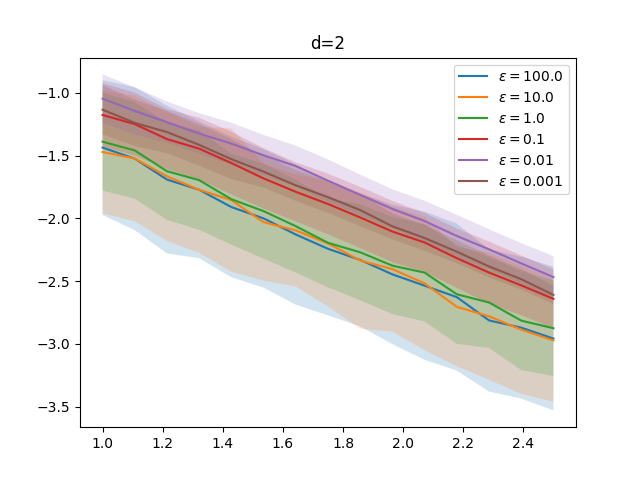} 
  \includegraphics[width=.32\textwidth]{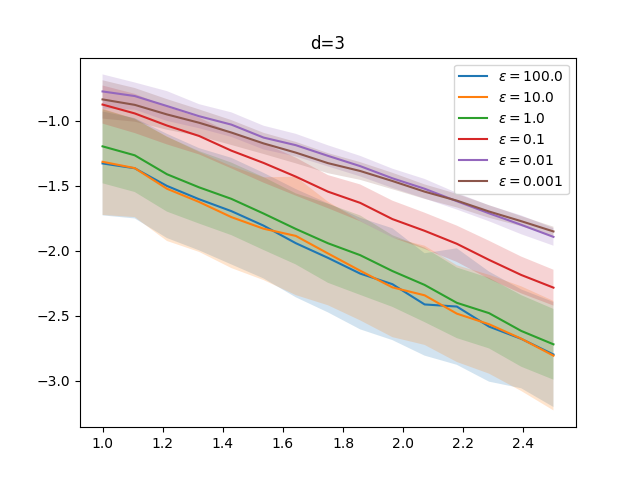}  
  \includegraphics[width=.32\textwidth]{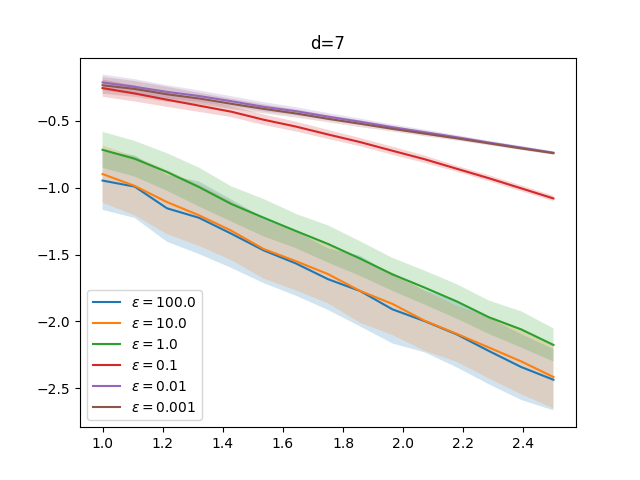} 
  \caption{$\bar W_\epsilon(\hat\alpha_n,\hat\alpha_n')$ as a function of $n$ in log-log space : Influence of $\epsilon$ for fixed $d$ on two uniform distributions on the hypercube with quadratic cost.} \label{fig:d}
  \vspace*{-10pt}
  \end{figure*}

 \begin{figure*}[h]
  \centering
  \includegraphics[width=.32\textwidth]{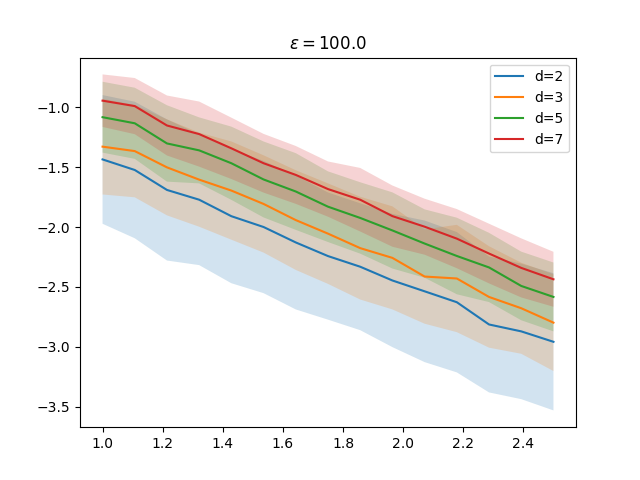} 
  \includegraphics[width=.32\textwidth]{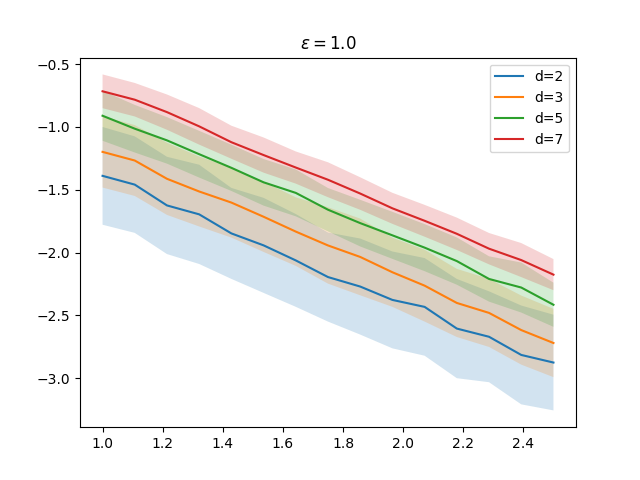}  
  \includegraphics[width=.32\textwidth]{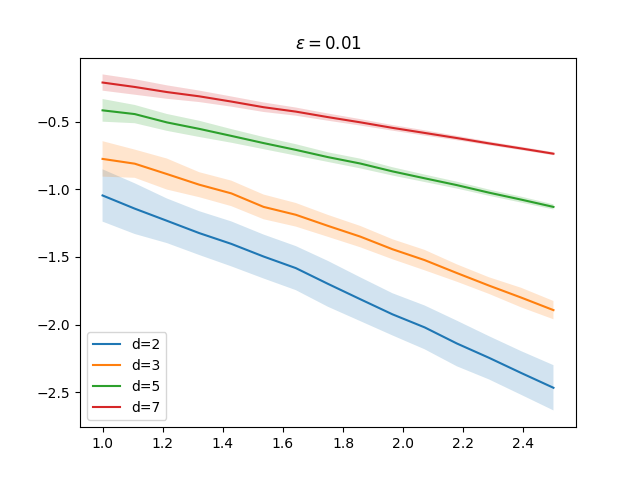} 
  \caption{$\bar W_\epsilon(\hat\alpha_n,\hat\alpha_n')$ as a function of $n$ in log-log space : Influence of $d$ for fixed $\epsilon$ on two uniform distributions on the hypercube with quadratic cost.} \label{fig:epsilon}
    \vspace*{-10pt}
  \end{figure*}

\begin{lem}\label{lem_lipschitz}
Let $\Aa = \{ (u,v) \mid  u \oplus v  \leq 2 L \abs{\Xx} + \norminf{c} \}$. We have:
\begin{itemize}
\item[(i)] the pairs of optimal potentials $(u^*,v^*)$ such that $u^*(0)=0$ belong to $\Aa$,
\item[(ii)] $f^\epsilon$ is B-Lipschitz in $(u,v)$ on $\Aa$ with $B \leq 1 + \exp(2\frac{L \abs{\Xx} + \norminf{c}}{\epsilon})$. 
\end{itemize}
\end{lem}
\begin{proof}
Let us prove that we can restrict ourselves to a subspace on which $f^\epsilon$ is Lipschitz in $(u,v)$.
$$f^\epsilon (u,v,x,y) = u(x) + v(y) - \epsilon  \exp\left(\frac {u(x)+v(y) - c(x,y)}{\epsilon}\right)$$
$$\nabla f^\epsilon (u,v) = 1 - \exp\left(\frac{u+v-c}{\epsilon}\right).$$

To ensure that $f^\epsilon$ is Lipschitz, we simply need to ensure that the quantity inside the exponential is upperbounded at optimality and then restrict the function to all $(u,v)$ that satisfy that bound.

Recall the bounds on the optimal potentials from Proposition~\ref{prop_potentials}. We have that $\forall x \in \Xx, y \in \Yy$,  
\eq{u(x)\leq L \abs{x} \quad \text{and} \quad v(y) \leq \max_x u(x) - c(x,y).}
Since we assumed $\Xx$ to be a bounded set, denoting by $\abs{\Xx}$ the diameter of the space we get that at optimality
$\forall x \in \Xx, y \in \Yy$
\eq{u(x) + v(y)  \leq 2 L \abs{\Xx} + \norminf{c}.}

Let us denote  $\Aa = \{ (u,v) \in (\mathbf{H}^s(\RR^d))^2 \mid  u \oplus v  \leq 2 L \abs{\Xx} + \norminf{c} \}$, we have that $\forall (u,v) \in \Aa$, 
\[\abs{\nabla f^\epsilon (u,v) }\leq 1 + \exp(2 \frac{L \abs{\Xx} + \norminf{c}}{\epsilon}).\qedhere\]
\end{proof}

We now have all the required elements to prove our sample complexity result on the Sinkhorn loss, by applying Proposition~\ref{prop:Bartlett}.

\begin{proof} (Theorem~\ref{sample_complexity_thm})
Since $f_\epsilon$ is Lipschitz and we are optimizing over $\mathbf H^s(\RR^d)$ which is a RKHS, we can apply Proposition~\ref{prop:Bartlett} to bound the $\sup$ in  Lemma~\ref{lem_bound}. We get:
\eq{\EE \vert W_\epsilon(\alpha,\beta) - W_\epsilon(\hat \alpha_n, \hat \beta_n) \vert \leq 3  \frac { 2 B \lambda} {n} \EE \sqrt{\sumin  k(X_i,X_i)}}
where $B \leq 1 + \exp(2\frac{L \abs{\Xx} + \norminf{c}}{\epsilon})$ (Lemma~\ref{lem_lipschitz}),
 $\lambda = O(\max(1,\frac{1}{\epsilon^{d/2}}))$ (Theorem~\ref{thm_rkhs}).
We can further bound $\sqrt{\sumin k(X_i,X_i)}$ by $\sqrt{n \max_{x \in \Xx} k(x,x)}$ where $k$ is the kernel associated to $H^s(\RR^d)$ (usually called Matern or Sobolev kernel) and thus $\max_{x \in \Xx} k(x,x) = k(0,0):= K$ which doesn't depend on $n$ or $\epsilon$.
Combining all these bounds, we get the convergence rate in $\frac{1}{\sqrt{n}}$ with different asymptotic behaviors in $\epsilon$ when it is large or small.
\end{proof}

Using similar arguments, we can also derive a concentration result:
\begin{cor}
With probability at least $1-\delta$,
\eq{\vert W_\epsilon(\alpha,\beta) - W_\epsilon(\hat \alpha_n, \hat \beta_n) \vert  \leq 6 B \frac{\lambda K}{\sqrt{n}} + C \sqrt{\frac{2 \log \frac{1}{\delta}}{n}}}
where $B,\lambda,K$ are defined in the proof above, and $C = \kappa + \epsilon \exp(\frac{\kappa}{\epsilon})$ with $\kappa = 2 L \abs{\Xx}+\norminf{c}$.
\label{cor:concentration}
\end{cor}

\begin{proof}
We apply the bounded differences (Mc Diarmid) inequality to $g: (x_1,\dots,x_n) \mapsto \sup_{u,v \in \Hh_\lambda^s}(\EE f_\epsilon^{XY} - \frac{1}{n}f_\epsilon^{X_i,Y_i})$. From Lemma~\ref{lem_lipschitz} we get that $\forall x,y$, $f_\epsilon^{xy}(u,v) \leq \kappa + \epsilon e^{\kappa/\epsilon}\eqdef C$, and thus, changing one of the variables in $g$ changes the value of the function by at most $2C/n$. Thus the bounded differences inequality gives 
\eq{
\PP\left(\abs{g(X_1,\dots ,X_n) - \EE g(X_1,\dots ,X_n)} >t \right) \leq 2 \exp(\frac{t^2n}{2C^2})}
Choosing $t  = C\sqrt{\frac{2 \log\frac{1}{\delta}}{n}}$ yields that with probability at least $1-\delta$
\eq{ g(X_1,\dots ,X_n) \leq \EE g(X_1,\dots ,X_n) + C \sqrt{\frac{2 \log \frac{1}{\delta}}{n}}
}
and from Theorem~\ref{sample_complexity_thm} we already have
\eq{\EE g(X_1,\dots ,X_n) = \EE \sup_{u,v \in \Hh_\lambda^s}(\EE f_\epsilon^{XY} - \frac{1}{n}f_\epsilon^{X_i,Y_i}) \leq  \frac { 2 B \lambda K} {\sqrt{n}}.}\end{proof}

 \begin{figure*}[h]
  \centering
  \includegraphics[width=.32\textwidth]{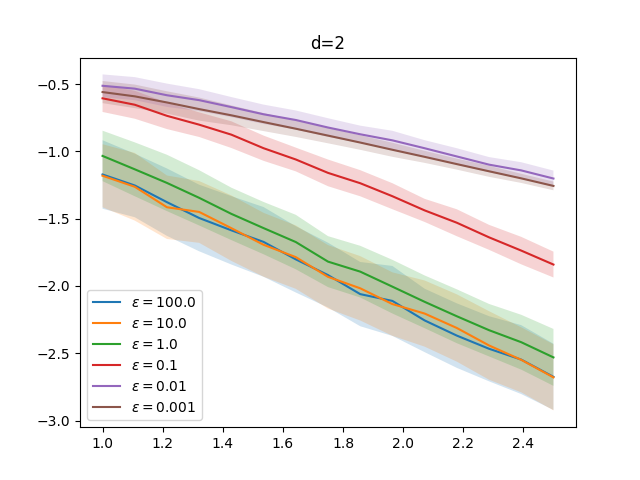} 
  \includegraphics[width=.32\textwidth]{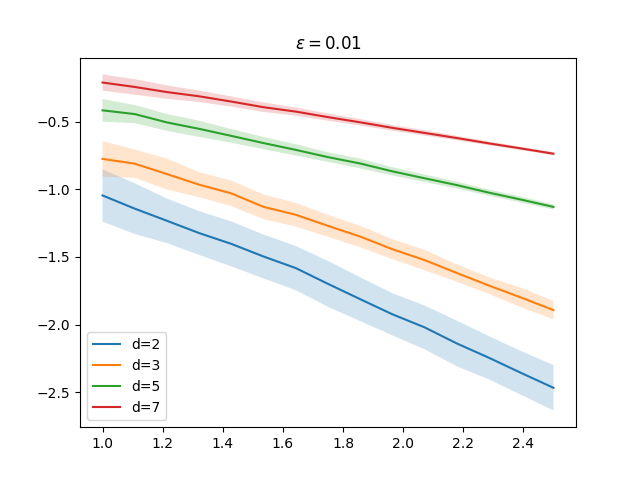}  
  \includegraphics[width=.32\textwidth]{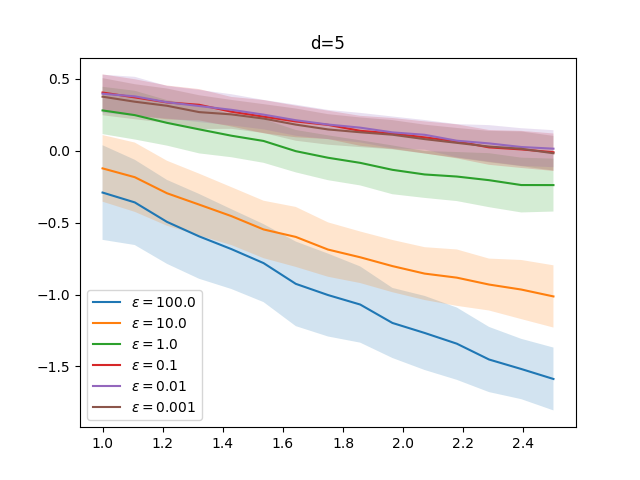} 
  \caption{$\bar W_\epsilon(\hat\alpha_n,\hat\alpha_n')$ as a function of $n$ in log-log space - cost $c(x,y)=\norm{x-y}_1$ with uniform distributions (two leftmost figures) and quadratic cost $c(x,y)=\norm{x-y}_2^2$ with standard normal distributions (right figure).} \label{fig:other}
  \end{figure*}

\section{Experiments} We conclude with some numerical experiments on the sample complexity of Sinkhorn Divergences. Since there are no explicit formulas for $W_\epsilon$ in general, we consider $\bar W_\epsilon(\hat\alpha_n,\hat\alpha_n ')$ where $\hat \alpha_n \eqdef \frac{1}{n} \sumin \delta_{Xi}$, $\hat \alpha_n' \eqdef \frac{1}{n} \sumin \delta_{Xi'}$ and $(X_1,\dots,X_n)$ and $(X_1',\dots,X_n')$ are two independent $n$-samples from $\alpha$. Note that we use in this section the normalized Sinkhorn Divergence as defined in~\eqref{SinkhornDiv}, since we know that $\bar W_\epsilon(\alpha,\alpha) =0$ and thus $\bar W_\epsilon(\hat\alpha_n,\hat\alpha_n') \rightarrow 0$ as $\n \rightarrow +\infty$ .

Each of the experiments is run $300$ times, and we plot the average of $\bar W_\epsilon(\hat\alpha_n,\hat\alpha_n')$ as a function of $n$ in log-log space, with shaded standard deviation bars.

First, we consider the uniform distribution over a hypercube with the standard quadratic cost $c(x,y) = \norm{x-y}_2^2$, which falls within our framework, as we are dealing with a $\Cc^\infty$ cost on a bounded domain.
Figure~\ref{fig:d} shows the influence of the dimension $d$ on the convergence, while Figure~\ref{fig:epsilon} shows the influence of the regularization $\epsilon$ on the convergence for a given dimension. The influence of $\epsilon$ on the convergence rate increases with the dimension: the curves are almost parallel for all values of $\epsilon$ in dimension 2 but they get further apart as dimension increases. As expected from our bound, there is a cutoff which happens here at $\epsilon =1$. All values of $\epsilon \geq 1$ have similar convergence rates, and the dependence on $\frac{1}{\epsilon}$ becomes clear for smaller values. The same cutoff appears when looking at the influence of the dimension on the convergence rate for a fixed $\epsilon$. The curves are parallel for all dimensions for $\epsilon \geq 1$ but they have very different slopes for smaller $\epsilon$. 

We relax next some of the assumptions needed in our theorem to see how the Sinkhorn divergence behaves empirically. First we relax the regularity assumption on the cost, using $c(x,y) = \norm{x-y}_1$. As seen on the two left images in figure \ref{fig:other} the behavior is very similar to the quadratic cost but with a more pronounced influence of $\epsilon$, even for small dimensions. The fact that the convergence rate gets slower as $\epsilon$ gets smaller is already very clear in dimension 2, which wasn't the case for the quadratic cost. The influence of the dimension for a given value of $\epsilon$ is not any different however.

We also relax the bounded domain assumption, considering a standard normal distribution over $\RR^d$ with a quadratic cost. While the influence of $\epsilon$ on the convergence rate is still obvious, the influence of the dimension is less clear. There is also a higher variance, which can be expected as the concentration bound from Corollary~\ref{cor:concentration} depends on the diameter of the domain.

For all curves, we observe that $d$ and $\epsilon$ impact variance, with much smaller variance for small values of $\epsilon$ and high dimensions. From the concentration bound, the dependency on $\epsilon$ coming from the uniform bound on $\f_\epsilon$ is of the form $\epsilon \exp(\kappa/\epsilon)$, suggesting higher variance for small values of $\epsilon$. This could indicate that our uniform bound on $f_\epsilon$ is not tight, and we should consider other methods to get tighter bounds in further work.

\section{Conclusion}
We have presented two convergence theorems for SDs: a bound on the approximation error of OT and a sample complexity bound for empirical Sinkhorn divergences. The $1/\sqrt{n}$ convergence rate is similar to MMD, but with a constant that depends on the inverse of the regularization parameter, which nicely complements the interpolation property of SDs pointed out in recent papers. Furthermore, the reformulation of SDs as the maximization of an expectation in a RKHS ball also opens the door to a better use of kernel-SGD for the computation of SDs.

Our numerical experiments suggest some open problems. It seems that the convergence rate still holds for unbounded domains and non-smooth cost functions. Besides, getting tighter bounds in our theorem might allow us to derive a sharp estimate on the optimal $\epsilon$ to approximate OT for a given $n$, by combining our two convergence theorems together.

\bibliographystyle{myabbrvnat}
\bibliography{biblio}
\normalsize

\onecolumn
\section*{Appendix}

\begin{proof} (Lemma~\ref{lem_pot_1})
For better clarity, we carry out the computations in dimension 1 but all the arguments are valid in higher dimension and we will clarify delicate points throughout the proof.

Differentiating both sides of the optimality condition \eqref{optim_condition} and rearranging yields
\begin{equation} \label{uprim}
u'(x) = \int c'(x,y) \gamma_\epsilon(x,y) \beta(y) dy.
\end{equation}

Notice that $\gamma_\epsilon'(x,y) = \frac{u'(x)-c'(x,y)}{\epsilon} \gamma_\epsilon(x,y)$. Thus by immediate recurrence (differentiating both sides of the equality again) we get that
\begin{equation}
u^{(n)}(x) =  \int g_n(x,y) \gamma_\epsilon(x,y) \beta(y) dy,
\end{equation}
where $g_{n+1} (x,y) = g_n' (x,y) + \frac{u'(x)-c'(x,y)}{\epsilon} g_n (x,y)$ and $g_1(x,y) = c'(x,y)$

To extend this first lemma to the $d$-dimensional case, we need to consider the sequence of indexes $\sigma = (\sigma_1,\sigma_2,\dots) \in \{1,\dots,d\}^{\NN}$ which corresponds to the axis along which we successively differentiate. Using the same reasoning as above, it is straightforward to check that
\eq{
\frac{\partial^k u}{\partial x_{\sigma_1}\dots\partial x_{\sigma_k}} = \int g_{\sigma,k} \gamma_{\epsilon}
}
where $g_{\sigma,1} = \frac{\partial c}{\partial x_{\sigma_1}}$ and $g_{\sigma,k+1} = \frac{\partial g_{\sigma,k+1}}{\partial x_{\sigma_{k+1}}} + \frac{1}{\epsilon} \left(  \frac{\partial u}{\partial x_{\sigma_{k+1}}} -  \frac{\partial c}{\partial x_{\sigma_{k+1}}} \right) g_{\sigma,k+1}$

\end{proof}

\begin{proof}(Lemma~\ref{lem_pot_2})
The proof is made by recurrence on the following property : \\
\textit{$P_n$ : For all $j=0,\dots,k $, for all $k=0,\dots,n-2$, $\norminf{g_{n-k}^{(j)}}$ is bounded by a polynomial  in $\frac{1}{\epsilon}$ of order $n-k+j-1$.} \\
Let us initialize the recurrence with $n=2$ 
\begin{eqnarray}
g_2 &= &g_1' + \frac{u'-c'}{\epsilon} g_1 \\
\norminf{g_2} &\leq & \norminf{g_1'} + \frac{\norminf{u'}+\norminf{c'}}{\epsilon} \norminf{g_1} 
\end{eqnarray}
Recall that  $\norminf{u'} = \norminf{g_1} = \norminf{c'}$.   Let $C = \max_{k} \norminf{c^{(k)}}$, we get that $\norminf{g_2}		\leq  C + \frac{C+C}{\epsilon} C$ which is of the required form.

Now assume that $P_n$ is true for some $n\geq2$. This means we have bounds on $g_{n-k}^{(i)}$, for $k=0,\dots,n-2$ and $i=0,\dots,k$. To prove the property at rank $n+1$ we want bounds on $g_{n+1-k}^{(i)}$, for $k=0,\dots,n-1$ and $i=0,\dots,k$. 
The only new quantity that we need to bound are  $g_{n+1-k}^{(k)}, k=0,\dots,n-1$.
Let us start by bounding $g_2^{(n-1)}$ which corresponds to $k=n-1$ and we will do a backward recurrence on $k$. By applying Leibniz formula for the successive derivatives of a product of functions, we get
\begin{eqnarray}
g_2 &=& g_1'+ \frac{u'-c'}{\epsilon} g_1 \\
g_2^{(n-1)} &=& g_1^{(n)} + \sum_{p=0}^{n-1} \binom{n-1}{p} \frac{u^{(p+1)}-c^{(p+1)}}{\epsilon} g_1^{(n-1-p)} \label{Leibniz}\\
\norminf{g_2^{(n-1)}} &\leq & \norminf{g_1^{(n)} }+ \sum_{p=0}^{n-1} \binom{n-1}{p} \frac{\norminf{u^{(p+1)}}+\norminf{c^{(p+1)}}}{\epsilon} \norminf{g_1^{(n-1-p)}} \\
&\leq & C + \sum_{p=0}^{n-1} \binom{n-1}{p} \frac{\norminf{g_{p+1}} + C }{\epsilon} C
\end{eqnarray}
Thanks to $P_n$ we have that  $\norminf{g_{p}}\leq \sum_{i=0}^{p} a_{i,p}\frac{1}{\epsilon^i} , p=1,\dots,n$ so the highest order term in $\epsilon$ in the above inequality is $\frac{1}{\epsilon^n}$. Thus we get
$\norminf{g_2^{(n-1)}} \leq  \sum_{i=0}^{n+1} a_{i,2,n-1}\frac{1}{\epsilon^i}$ which is of the expected order 

Now assume $g_{n+1-j}^{(j)}$ are bounded with the appropriate polynomials for $j < k\leq n-1$. Let us bound $g_{n+1-k}^{(k)}$
\begin{eqnarray}
\norminf{g_{n+1-k}^{(k)}} &\leq & \norminf{g_{n-k}^{(k+1)} }+ \sum_{p=0}^{k} \binom{k}{p} \frac{\norminf{u^{(p+1)}}+\norminf{c^{(p+1)}}}{\epsilon} \norminf{g_{n-k}^{(k-p)}} \\
								&\leq & \norminf{g_{n-k}^{(k+1)} }+ \sum_{p=0}^{k} \binom{k}{p} \frac{\norminf{g_{p+1}}+C}{\epsilon} \norminf{g_{n-k}^{(k-p)}}
\end{eqnarray}
The first term $\norminf{g_{n-k}^{(k+1)} }$  is bounded with a polynomial of order $\frac{1}{\epsilon^{n+1}}$ by recurrence assumption. Regarding the terms in the sum, they also have all been bounded and 
$$\norminf{g_{p+1}}\norminf{g_{n-k}^{(k-p)}} \leq \left( \sum_{i=0}^{p} a_{i,p+1}\frac{1}{\epsilon^i} \right) \left(\sum_{i=0}^{n-p} a_{i,n-k,k-p}\frac{1}{\epsilon^i} \right) \leq \sum_{i=0}^{n} \tilde a_i \frac{1}{\epsilon^i}$$

So $\norminf{g_{n+1-k}^{(k)}}  \leq \sum_{i=0}^{n+1} a_{i,n+1-k,k}\frac{1}{\epsilon^i} $

To extend the result in $\RR^d$, the recurrence is made on the the following property
\eql{
\norminf{g_{\sigma,n-k}^{(j)}} \leq \sum_{i=0}^{n-k+\abs{j}-1} a_{i,n-k,j,\sigma}\frac{1}{\epsilon^i} \qquad \forall j \mid \abs{j} = 0,\dots,k \quad \forall k=0,\dots,n-2 \quad \forall \sigma  \in \{ 1,\dots,d \}^{\NN}
}
where $j$ is a multi-index since we are dealing with multi-variate functions, and $g_{\sigma,n-k}$ is defined at the end of the previous proof. The computations can be carried out in the same way as above, using the multivariate version of Leibniz formula in \eqref{Leibniz} since we are now dealing with multi-indexes.
\end{proof}



\end{document}